\documentclass[11pt]{amsart}
\usepackage{amsmath, amsthm, amscd, amsfonts, amssymb}
\usepackage{graphicx}\usepackage{color}\usepackage[colorlinks]{hyperref}

 \usepackage[all]{xy}

\begin{document}
\newtheorem{theorem}{Theorem}[section]
\newtheorem{proposition}[theorem]{Proposition}
\newtheorem{algorithm}[theorem]{Algorithm}
\newtheorem{axiom}[theorem]{Axiom}
\newtheorem{case}[theorem]{Case}
\newtheorem{claim}[theorem]{Claim}
\newtheorem{conclusion}[theorem]{Conclusion}
\newtheorem{condition}[theorem]{Condition}
\newtheorem{conjecture}[theorem]{Conjecture}
\newtheorem{corollary}[theorem]{Corollary}
\newtheorem{criterion}[theorem]{Criterion}
\newtheorem{lemma}[theorem]{Lemma}
\newtheorem{counterexample}[theorem]{Counterexample}
\theoremstyle{definition}
\newtheorem{definition}[theorem]{Definition}
\newtheorem{example}[theorem]{Example}
\newtheorem{exercise}[theorem]{Exercise}
\newtheorem{notation}[theorem]{Notation}
\newtheorem{note}[theorem]{Note}
\newtheorem{problem}[theorem]{Problem}
\newtheorem{remark}{Remark}
\newtheorem{solution}[theorem]{Solution}
\newtheorem{summary}[theorem]{Summary}
\newtheorem{construction}[theorem]{Construction}
\numberwithin{equation}{section}
\makeatletter
\newcommand{\leqnomode}{\tagsleft@true}
\newcommand{\reqnomode}{\tagsleft@false}

 \title[On the radical and torsion theory in the category of $S$-acts ] {On radical and torsion theory in the category of $S$-acts}
\author[ M. Haddadi and S.M. N. Sheykholislami]{ M. Haddadi and S.M. N. Sheykholislami}
\address{\rm{Department of Mathematics, Statistics and Computer Sciences, Semnan University, Semnan, Iran.}}
\email{ m.haddadi@semnan.ac.ir, haddadi$\_$1360@yahoo.com}

\begin{abstract}
In abelian categories like the category of $R$-modules and even in the category {\bf S-Act}$_{0}$ of S-acts
with a unique zero, idempotent radicals and torsion theories are equivalent, and the $\tau$-torsion and $\tau$-torsion free classes of a torsion theory $\tau$ are closed under coproducts. These are not
necessarily true in the category {\bf S-Act} of $S$-acts. In this paper, we prove that torsion theories are equivalent with the Kurosh-Amitsur radicals. We, also, show that the class of Kurosh-Amitsur radicals is a reflective subcategory of Hoehnke radicals, as a poset.
\end{abstract}

\maketitle

\section{Introduction}
The importance of Radical and Torsion theory in many areas of mathematics is well known. These topics are intensively studied throughout the years and developed customarily in abelian groups, semigroup, modules, and even abelian categories (see \cite{Clifford, Dickson, Hoehnke 1965-2, Schelter}). Here we are going to study on these topics in the category {\bf S-Act} of $S$-acts.

In this paper after recapitulating the rudiments of the Kurosh-Amitsur radical and torsion theory in the category {\bf S-Act} of $S$-acts, we demonstrate the well known correspondence between radicals and torsion theory in this category. Indeed, we show that every Kurosh-Amitsur radical is given by a pair of the appropriately chosen classes of $S$-acts. We also show that the class of Kurosh-Amitsur radicals
as a poset is a reflective subcategory of the class of  Hoehnke radicals, see
Section \ref{T. sub H.}.

Now let us recall some necessary notions needed throughout the paper.

 An {\em $S$-act} $A$ over a monoid $S$ is a set $A$ together with an
action $(s,a)\mapsto sa$, for $a\in A$, $s\in S$, subject
to the rules $s(ta) = (st)a$ and $1a=a$, where $1$ is the
identity element of the monoid $S$ and $a\in A$ and $s, t\in
S$. We will work in the category of all $S$-acts and all
homomorphisms $f: A\to B$, subject to $f(as) = f(a)s$, for all $a\in A$
and $s\in S$. An element $z$ of an $S$-act $A$ is said to be a \textit{zero} if $sz = z$, for
all $s \in S$. Also, we say that an $S$-act $A$ is \textit{trivial} if $|A|\leq
1$.

 An equivalence relation $\rho$ on an $S$-act $A$ is called a {\em
congruence} on $A$ if $a\rho a'$ implies $(as)\rho(a's)$, for
all $s\in S$. We denote the set of all congruences on $A$ by
$\rm{Con}(A)$, which forms a lattice, see \cite{Burris}. In the lattice $\rm{Con}(A)$
there is the smallest congruence,
the diagonal relation $\Delta_A = \{(a, a) | a \in A\}$, and the
largest congruence, the total relation $\nabla_A = \{(a, b) |\ a, b \in
A\}$. Every congruence $\rho\in \rm{Con}(A)$ determines a
partition of $A$ into $\rho$-cosets and a system $\Sigma_\rho$ of
those $\rho$-cosets each of which is a non-trivial subact of $A$. Of course, $\Sigma_\rho$
may be empty. Throughout this paper we use the general Rees
congruence introduced in \cite{Wiegandt}; that is, in a
{\em Rees congruence} the cosets are either  subacts or consist of
one element. Also every system $\Sigma$ of disjoint non-trivial
subacts of an $S$-act $A$ determines a Rees congruence
$\rho_\Sigma$ given by
\[
(a, b)\in \rho_\Sigma \Longleftrightarrow\begin{cases}
a, b \in B &\ \text{for some} \ B\in \Sigma\\
a = b&\ \text{otherwise}.
\end{cases}
\]

We call $\rho_{\Sigma}$ ($\rho_{B}$ if $\Sigma=\{B\}$) a \textit{generated Rees congruence by $\Sigma$} on $A$ and $A/\rho_\Sigma$ a {\em Rees factor} of $A$ over $\rho_{\Sigma}$ (or for short, the Rees factor). Clearly, there is
a one-to-one correspondence between Rees congruences and systems
of disjoint non-trivial subacts. Moreover, the set of all systems
of disjoint subacts of an $S$-act $A$ forms a lattice isomorphic
to the sublattice of all Rees congruences in $\rm{Con}(A)$. Every congruence $\chi\in
\rm{Con}(A)$ determines a Rees congruence $\rho_\Sigma$ via
$\Sigma_\chi$, with $\rho_\Sigma\leq\chi$.

 A congruence $\chi_B$ of a subact $B$ of an $S$-act $A$ can be extend to a congruence of the $S$-act
$A$. There is always the smallest extension $\chi_A$ given by
\[
(a, b)\in\chi_A\Longleftrightarrow\begin{cases}
(a, b) \in\chi_B\\
a = b \ \text{otherwise}.
\end{cases}
\]
Therefore we may consider each congruence $\chi_B\in\rm{Con}(B)$ as a congruence of
$\rm{Con}(A)$ by identifying $\chi_B$ and $\chi_A$. In particular, $\nabla_B$ can be viewed as the Rees
congruence $\rho_B \in\rm{ Con}(A)$ determined by the system $\Sigma_{\rho_B} = \{B\}$.

 Whenever talking  about a subclass $\mathbb{C}$ of $S$-acts, we assume that $\mathbb{C}$ is
closed under taking isomorphic copies and $\mathbb{C}$
contains of all trivial subacts.

 Given a subclass $\mathbb{C}$ of $S$-acts, a system $\Sigma$ of disjoint non-trivial subacts of an $S$-act $A$ is called a $\mathbb{C}$-$system$ if $B \in \mathbb{C}$, for each $B \in\Sigma$.

In the sequel of this paper we frequently use the closedness of a subclass $\mathbb{C}$ of $S$-acts under a special property such as  closedness under  homomorphic image, closedness under congruence extensions, closedness under Rees extensions, closedness under subact, closedness under product, and inductive property which have defined in  \cite{Wiegandt}.

\medskip

Although  the radical and the torsion theory for $S$-acts were introduced and investigated by R. Wiegandt \cite{Wiegandt}, but  it seems necessary to define the radical in a more general manner. Here we follow the category theoretical view of radical \cite{Tholen} and give the following definition of Hoehnke radical in {\bf S-Act} which may also is called a normal Hoehnke radical.    

 \begin{definition}\label{Hoehnke}
$(1)$ A \textit{normal Hoehnke radical} (or simply a \textit{Hoehnke radical}) is an assignment $r: A \to r(A)$,
assigning to each $S$-act $A$ a congruence $r(A) \in \rm{Con}(A)$ in such a way that 

\rm{(i)} $r$ is functorial, or more precisely, every homomorphism $f: A\to B$ induces the homomorphism $r(f):r(A)\to r(B)$. Meaning that   $(f(a),f(a'))\in r(B)$ if $(a,a')~\in ~r(A)$, for every homomorphism $f:A\rightarrow B$. Note that $r(A)$ and $r(B)$ are, respectively, subacts of $A\times A$ and $B\times B$, since $r(A)\in \rm{Con}(A), r(B)\in \rm{Con}(B)$, and

 \rm{(ii)} $r(A/r(A)) = \Delta_{A/r(A)}$.
\medskip

$(2)$ A  Hoehnke radical $r$ is said to be \emph{hereditary}, if $ r(B) = r(A)\wedge \nabla_{B}$
for every S-acts $A$ and $B \leq A$.
\medskip
\end{definition}
 With every  Hoehnke radical $r$ one can associate two classes of
$S$-acts, namely {\it radical class} $\mathbb{R}_r$ and {\it
semisimple class} $\mathbb{S}_r$, as follows:
\begin{align*}
\mathbb{R}_{r} = \{A \ |\ r(A) = \nabla_{A}\}\\
\mathbb{S}_{r} = \{A \ |\ r(A) = \Delta_{A}\}.
\end{align*}

 It is worth noting that $\mathbb{S}_{r}$ is closed under taking
subacts and products. Indeed, Since every (normal) Hoehnke radical is a Hoehnke radical in the sense of \cite{Wiegandt}, $\mathbb{S}_{r}$ is closed under products, also for every subact $B$ of a semisimple $S$-act $A$, by Definition \ref{Hoehnke}, the embedding map from $B$ to $A$ implies $r(B)\subseteq r(A)=\Delta_{A}$. Hence $r(B)=\Delta_B$.

Every subclass $\mathbb{S}$ of $S$-acts which is closed under taking subacts and products, determines a
 Hoehnke radical $r_{\mathbb{S}}$ defined by:
\begin{align*}
r_{\mathbb{S}}(A)=\bigwedge (\chi \in \rm{Con}(A) \ |\ A/ \chi \in \mathbb{S}).
\end{align*}
Moreover, $\mathbb{S} = \mathbb{S}_{r}$ if and only if $r = r_{\mathbb{S}}$ .

\begin{definition}  A  Hoehnke radical $r$ of $S$-acts is called a {\em Kurosh-Amitsur}
radical, if

 \rm{(i)} r(A) is a Rees congruence, for all $S$-acts $A$,

\rm{(ii)} for every $B\in \Sigma_{r(A)}$, $r(B)=\nabla_B$,

 \rm{(iii)} if $\Sigma$ is an $\mathbb{R}_{r}$-$system$ of disjoint non-trivial subacts of an S-act A, then
$\Sigma\leq \Sigma_{r(A)} $, that is, for every $B\in \Sigma$, there exists $C\in \Sigma_{r(A)}$ with $B\leq C$.
\end{definition}
\medskip

We recall, from \cite{Wiegandt}, that a subclass
$\mathbb{S}$ of $S$-acts is a semisimple class of a Kurosh-Amitsur radical $r$
if and only if
\begin{enumerate}
\item $\mathbb{S}$ is closed under taking subacts,
\item $\mathbb{S}$ is closed under taking products,
\item $\mathbb{S}$ is closed under taking congruence extensions.
\end{enumerate}

 Also a subclass $\mathbb{R}$ of $S$-acts is a radical class of
a radical $r$ if and only if
\begin{enumerate}
\item $\mathbb{R}$ is homomorphically closed,
\item $\mathbb{R}$ has the inductive property,
\item $\mathbb{R}$ is closed under Rees extensions.
\end{enumerate}
Furthermore,
\begin{align*}
r(A) = \vee\{\rho \in \rm{Con}(A) \ | \ \rho \ \text{is a Rees congruence and} \ \Sigma_\rho \subseteq\mathbb{R}\}.
\end{align*}

 The readers may consult \cite{Adamek, Burris, Kilp} for general facts about category theory and universal algebra used in this paper. Here we also follow the notations and terminologies used there.

 \medskip
\section{Torsion theories as a reflective subcategory of  Hoehnke radicals}\label{T. sub H.}
 As is well known, in the category of $R$-modules there is a bijective correspondence between torsion theories and idempotent radicals,  radicals subject to the rule  $r\circ r=r$, (see e.g. \cite{Stenstrom}). In this section first, we define the torsion theory in the category of $S$-acts and we show that 
 there exists a bijective correspondence between torsion theories and Kurosh-Amitsur radicals and
  then, we prove that the class of Kurosh-Amitsur radicals over $A$ is a reflective full subcategory of the class of  Hoehnke radicals over $A$, for every $S$-act $A$.

\begin{definition}\label{torsion theory}
A pair $\tau=(\mathbb{T}, \mathbb{F})$ of subclasses of $S$-acts is called a \textit{torsion theory} if it satisfies the following conditions:
\begin{enumerate}
\item ${\rm Hom}(A, B)$ is empty or when $B$ has zeros, ${\rm Hom}(A, B)$ consists of the zero homomorphisms, for every $A\in \mathbb{T}$ and $B\in \mathbb{F}$.
\item If, for every $B\in \mathbb{F}$, ${\rm Hom}(A, B)$ is empty or when $B$ has zeros, ${\rm Hom}(A, B)$ consists of the zero homomorphisms then $A\in \mathbb{T}$.
\item If, for every $A\in\mathbb{T}$, ${\rm Hom}(A, B)$ is empty or when $B$ has zeros, ${\rm Hom}(A, B)$ consists of the zero homomorphisms then $B\in \mathbb{F}$.
\end{enumerate}
\end{definition}

Let us call the class $\mathbb{T}$ as the \textit{torsion class} of $\tau$ and its members be called $\tau$-\textit{torsion} acts, whereas the class $\mathbb{F}$ is called the \textit{torsion-free class} of $\tau$ and its members are called $\tau$-\textit{torsion-free} acts, as it is used in module theory.

Although the assertion of the following lemma and Theorem 2.2 \cite{Wiegandt} are same but they express different  statements, since our  definition of a pair of radical and semisimple class $(\mathbb{R},\mathbb{S})$ is more general, see the following lemma.

\begin{lemma}\label{pair kurosh-amitsur}
A pair $(\mathbb{R}, \mathbb{S})$ of subclasses of $S$-acts is the radical class and the semisimple class of a Kurosh-Amitsur radical $r$ if and only if
\begin{enumerate}
\item $\mathbb{R}\cap \mathbb{S}$ consists of trivial $S$-acts,
\item $\mathbb{R}$ is homomorphically closed,
\item $\mathbb{S}$ is closed under taking subacts,
\item every $S$-act $A$ has an $\mathbb{R}$-$system$  $\Sigma$ whose Rees factor, $A/\rho_{_{\Sigma}}$, belongs to $\mathbb{S}$.
\end{enumerate}
\end{lemma}

\begin{proof}
To prove necessity, let $r$ be a Kurosh-Amitsur radical. Then
\[\mathbb{R}_{r}=\{A\ |\ A\text{ has no non-trivial homomorphic image in } \mathbb{S}_{r}\}\]
by Proposition 2.3 of \cite{Wiegandt}, and
\[\mathbb{S}_{r}=\{A\ | \ B\subseteq A \text{ and } B\in \mathbb{R} \text{ imply } |B|\leq 1\}\]
by Theorem 2.4 of \cite{Wiegandt}. So $(\mathbb{R}_{r},\mathbb{S}_{r})$ satisfies the conditions (1-4), by Theorem 2.6 of \cite{Wiegandt}.

 To prove sufficiency, first we show that $\mathbb{R}$ is a radical class of a Kurosh-Amitsur radical such as $r$. To do so, it is enough to prove that:
\begin{enumerate}
\item[($a$)]$\mathbb{R}$ is homomorphically closed,
\item[($b$)]$\mathbb{R}$ has the inductive property, and
\item[($c$)]$\mathbb{R}$ is closed under Rees extensions,
\end{enumerate}
then Theorem 2.4 of \cite{Wiegandt}gives the result.

Part $(a)$ is established by the second property of the hypothesis.
To prove part $(b)$ given on ascending chain $\{A_{i}\}_{i\in I}$, consider  the associated $\mathbb{R}$-$system$  with $\bigcup_{i\in I} A_{i}$ as $\Sigma$. Then we have $\bigcup_{i\in I} A_{i}/\rho_{_{\Sigma}}\in \mathbb{S}$, by Condition(4) of the hypothesis. Now, since $A_{i}/(\rho _{_{\Sigma}}\wedge \nabla_{A_{i}})$ is a subact of $\bigcup_{i\in I} A_{i}/\rho_{_{\Sigma}}$, Condition(3) implies that $A_{i}/(\rho _{_{\Sigma}}\wedge \nabla_{A_{i}})\in\mathbb{S}$, for every $i\in I$. Also, by Condition(2) of the hypothesis, $A_{i}\in \mathbb{R}$ implies that $A_{i}/(\rho_{_{\Sigma}}\wedge \nabla_{A_{i}})\in \mathbb{R}$, for every $i\in I$. Hence $A_{i}/(\rho_{_{\Sigma}}\wedge \nabla_{A_{i}})\in \mathbb{R}\cap \mathbb{S}$ and  Condition(1) of the hypothesis indicates that $A_{i}/(\rho_{_{\Sigma}}\wedge \nabla_{A_{i}})$ is a trivial $S$-act. Therefore, $\nabla_{A_{i}} \leq \rho_{_{\Sigma}}$ for all $i\in I$. Thus there exists $B\in \Sigma$ such that $A_{i}\leq B$, for all $i\in I$. Now since $\{A_{i}\}_{i\in I}$ is an ascending chain and $\Sigma $ is a system of disjoint subacts of the $S$-act $\bigcup _{i\in I}A_{i}$, we have $\bigcup _{i\in I} A_{i}\leq B\leq \bigcup_{i\in I}A_{i}$. Therefore, $\bigcup _{i\in I}A_{i}=B$ and this means that $\bigcup_{i\in I}A_{i}\in \mathbb{R}$.

 To prove $(c)$ we show that for every $S$-act $A$ and every Rees congruence $\chi$ on $A$ with $\Sigma_{\chi}\subseteq \mathbb{R}$ and $A/\chi\in \mathbb{R}$, $A\in \mathbb{R}$. To do so, it is enough to show that, for the associated $\mathbb{R}$-$system$ with $A$ such as $\Sigma$, $A/\rho_{\Sigma}$ is a trivial $S$-act. Because,  $A/\rho_{\Sigma}$ being a  trivial $S$-act implies that $\rho_{\Sigma}$ is the total relation on $A$ and this means that $\Sigma =\{A\}$. But $\Sigma$ is an $\mathbb{R}$-$system$, so $A\in \mathbb{R}$. To prove that $A/\rho_{\Sigma}$ is trivial, we prove $A/\rho_{\Sigma}\in \mathbb{S}\cap \mathbb{R}$. But $A/\rho_{\Sigma}\in \mathbb{S}$, by Condition(4) of the hypothesis. Now we claim that $A/\rho_{\Sigma}\in \mathbb{R}$. Since $B/(\rho_{\Sigma}\wedge \nabla_{B})$ is a subact of $A/\rho_{\Sigma}$, for every $B\in \Sigma_{\chi}$,  Condition(3) of the hypothesis indicates $B/(\rho_{\Sigma}\wedge \nabla_{B})\in \mathbb{S}$. We also have $B\in \Sigma_\chi$ and $\Sigma_{\chi}\subseteq \mathbb{R}$, so $B\in \mathbb{R}$. Hence $B/(\rho_{\Sigma}\wedge \nabla_{B})\in \mathbb{R}$ follows from Condition(2) of the hypothesis. Thus $B/(\rho_{\Sigma}\wedge \nabla_{B})\in \mathbb{R}\cap \mathbb{S}$ and Condition(1) implies that $B/(\rho_{\Sigma}\wedge \nabla_{B})$ is a trivial $S$-act. That is, $\nabla_{B}\leq \rho_{_{\Sigma}}$, for all $B\in \chi$, and this means $\Sigma_{\chi}\leq \Sigma$. Now, by considering the canonical epimorphism $\pi :A/\chi \longrightarrow A/\rho_{_{\Sigma}}$, we have $A/\rho_{_{\Sigma}}\in \mathbb{R}$ which means that $A/\rho_{_{\Sigma}}$ is a trivial $S$-act. So $A\in \mathbb{R}$ and this indicates the closedness of $\mathbb{R}$ under Rees extension.

 Now the desired result follows from Corollary 2.5 and Theorem 2.6 of \cite{Wiegandt}.
\end{proof}

 \begin{theorem}
 Torsion theories of acts are the same as pairs of corresponding radical and semisimple classes of Kurosh--Amitsur radicals.
\end{theorem}

\begin{proof}
 Let $(\mathbb{T},\mathbb{F})$ be a torsion theory. We show that it satisfies properties (1)--(4) of Lemma \ref{pair kurosh-amitsur}. 

Properties (1)--(3) follow immediately from the definition of a torsion theory. To prove property (4), take an arbitrary $S$-act $A$ and denote by $t(A)$ the largest Rees congruence over $A$ whose non-singleton classes are in $\mathbb{T}$. Denote by $\pi$ the canonical homomorphism $\pi :A\to A/t(A)$. Take an arbitrary $B\in \mathbb{T}$ and any homomorphism $f: B\to A/t(A)$. We show that $\pi^{-1}(f(B))\in \mathbb{T}$. Take an arbitrary $X\in \mathbb{F}$ and a homomorphism 
$$g: \pi^{-1}(f(B))\to X,$$ 
then $\ker (\pi)\cap \nabla_{\pi^{-1}(f(B))} \subseteq \ker (g)$ must hold. Indeed, if $x\neq y$ and 
$(x,y)\in \ker (\pi)\cap \nabla_{\pi^{-1}(f(B))}$ then, by $(x,y)\in \ker (\pi)$, there exists a $C\in \Sigma_{t(A)} \subseteq\mathbb{T}$ such that $(x,y)\in C$. Now $C\in\mathbb{T}$ and $X\in\mathbb{F}$, so $g|_C : C\to X$ is the trivial homomorphism, and thus 
$\ker (\pi)\cap \nabla_{\pi^{-1}(f(B))} \subseteq \ker (g)$. Therefore $\bar g : f(B)\to X\,:\, f(b)\mapsto g(\pi^{-1}(f(b)))$ is well defined and it is the trivial homomorphism, whence $g$ is also the trivial homomorphism. Thus $\pi^{-1}(f(B))\in \mathbb{T}$. So there exists a $C\in\Sigma_{t(A)}$ such that 
$$f(B)\leq C/(t(A)\cap\nabla_C) = \{[C]_t\},$$
 whence $|f(B)|\leq 1$ and therefore $A/t(A)\in\mathbb{F}$.

For the converse, we prove that the pair $(\mathbb{R}_r, \mathbb{S}_r)$ is a torsion theory, for every Kurosh--Amitsur radical $r$. To do so, we show that $(\mathbb{R}_{r},\mathbb{S}_{r})$ satisfies the conditions of Definition \ref{torsion theory}. Indeed, by Proposition 2.3 of \cite{Wiegandt},
\reqnomode
\begin{align}\label{eq1}\tag{I}
\mathbb{R}_{r} =\{A\ |\ A \text{ has no non-trivial homomorphic image in } \mathbb{S}\}.
\end{align}
So ${\rm Hom}(A, B)$ is empty or when $B$ has zeros, ${\rm Hom}(A, B)$ consists of the zero homomorphisms, for every $A\in \mathbb{R}$ and $B\in \mathbb{S}$. Also (\ref{eq1}) indicates that $A\in \mathbb{R}$ when ${\rm Hom}(A, B)$ is empty or consists of the zero homomorphisms, for every $B\in \mathbb{S}$. So $(\mathbb{R}_{r},\mathbb{S}_{r})$ satisfies the first and the second properties of Definition \ref{torsion theory}. To prove the third property of Definition \ref{torsion theory}, let $B$ be an $S$-act with no non-trivial homomorphism from $A$ to $B$, for every $A\in \mathbb{R}$. Then no subact of $B$ can  belong to $\mathbb{R}$. Now, since
\[\mathbb{S}_{r}=\{A\ |\ B\subseteq A \text{ and } B\in \mathbb{R} \text{ imply } |B|\leq 1\},\]
by Theorem 2.4 of \cite{Wiegandt}, $B\in \mathbb{S}$. That is $(\mathbb{R}_{r},\mathbb{S}_{r})$ satisfies the third property of Definition \ref{torsion theory} and hence $(\mathbb{R}_{r}, \mathbb{S}_{r})$ is a torsion theory.
\end{proof}

\begin{lemma}
Let $\mathbb{C}$ be a subclass of $S$-acts which is closed under taking subacts and products. Then the radical class of the  Hoehnke radical $\mathbb{R}_{r_{\mathbb{C}}}$ has the following properties.
\begin{enumerate}
\item The class $\mathbb{R}_{r_{\mathbb{C}}}$ is homomorphically closed.
\item The class $\mathbb{R}_{r_{\mathbb{C}}}$ has the inductive property.
\item The class $\mathbb{R}_{r_{\mathbb{C}}}$ is closed under Rees congruence extension.
\end{enumerate}
\end{lemma}
\begin{proof}
(1) With the definition of a radical class in mind, this assertion is clear.

 (2) To prove the inductive property, consider an ascending chain $\{A\}_{i}$ in $\mathbb{R}_{r_{\mathbb{C}}}$. Then we have the canonical homomorphism
\[ \pi_{i}:A_{i}\longrightarrow \frac{\bigcup _{i\in I}A_{i}}{r_{\mathbb{C}}(\bigcup_{i\in I}A_{i})}\]
for every $i\in I$. We should note that $ \frac{\bigcup _{i\in I}A_{i}}{r_{\mathbb{C}}(\bigcup_{i\in I}A_{i})} \in \mathbb{C}$, since the semisimple class of $r_{\mathbb{C}}$ is exactly $\mathbb{C}$. Also $\pi_{i}(A_{i})\leq \frac{\bigcup _{i\in I}A_{i}}{r_{\mathbb{C}}(\bigcup_{i\in I}A_{i})}$, and so $\pi_{i}(A_{i})\in \mathbb{C}$. Now, since $A_{i}\in \mathbb{R}_{r_{\mathbb{C}}}\subseteq \mathcal{R}\mathbb{C}$, where
\[\mathcal{R}\mathbb{C}=\{A\ | \ A\text{ has no non-trivial homomorphic image in }\mathbb{C}\}, \]
we have that $\pi_{i}(A_{i})$ is a trivial $S$-act. Hence $\pi_{i}(x)=\pi_{i}(y)$, for every $x, y\in A_{i}$. So the canonical homomorphism
\[\pi:\bigcup _{i\in I}A_{i}\longrightarrow \frac{\bigcup_{i\in I}A_{i}}{r_{\mathbb{C}}(\bigcup_{i\in I}A_{i})}\]
maps every $x, y\in \bigcup_{i\in I}A_{i}$ to the same element in $\frac{\bigcup_{i\in I}A_{i}}{r_{\mathbb{C}}(\bigcup_{i\in I}A_{i})}$.
Indeed, since $(A_{i})_{i\in I}$ is an ascending chain, for every $x, y\in \bigcup_{i\in I} A_{i}$, there exists $j\in I$ such that $x, y\in A_{j}$ and $\pi(x)=\pi_{j}(x)=\pi_{j}(y)=\pi(y)$. But the canonical homomorphism $\pi$ is onto, so $\frac{\bigcup_{i\in I}A_{i}}{r_{\mathbb{C}}(\bigcup_{i\in I}A_{i})}$ is a trivial $S$-act. That is, $r_{\mathbb{C}}(\bigcup_{i\in I}A_{i})=\nabla_{\bigcup_{i\in I}A_{i}}$. Hence $\bigcup_{i\in I}A_{i}\in \mathbb{R}_{r_{\mathbb{C}}}$.

 (3) Now we show that $\mathbb{R}_{r_{\mathbb{C}}}$ is closed under Rees extension. Let $\rho$ be a Rees congruence with $A/\rho\in \mathbb{R}_{r_{\mathbb{C}}}$ and $\Sigma_{\rho}\subseteq \mathbb{R}_{r_{\mathbb{C}}}$. We show that $A\in \mathbb{R}_{r_{\mathbb{C}}}$. \\
Otherwise, $A/r_{\mathbb{C}}(A)\in \mathbb{C}$ is a non-trivial $S$-act. But, since for every $B\leq A$, $r_{\mathbb{C}}(B)\leq r_{\mathbb{C}}(A)\wedge \nabla_{B}$, we have $\rho\leq r_{\mathbb{C}}(A)$. Hence we get the epimorphism
\[\begin{matrix}
A/\rho &\longrightarrow &A/r_{\mathbb{C}}(A)\\
a/\rho &\mapsto & a/r_{\mathbb{C}}(A).
\end{matrix} \]
That is, $A/\rho \in \mathbb{R}_{r_{\mathbb{C}}}$ has a non-trivial homomorphic image in $\mathbb{C}$, which is a contradiction. So $A\in \mathbb{R}_{r_{\mathbb{C}}}$.
\end{proof}

\begin{remark}\label{r1}
By the above lemma, given a subclass $\mathbb{S}$ of $S$-acts which is closed under taking subacts and products, we get the  Hoehnke radical $r_\mathbb{S}$ whose radical class $\mathbb{R}_{r_{\mathbb{S}}}$ has the desired property of Theorem 2.4 in \cite{Wiegandt}. Hence $\mathbb{R}_{r_{\mathbb{S}}}$, by that theorem, gives a Kurosh-Amitsur radical $r_{k}$. Let us denote the semisimple class of $r_{k}$ by $\mathbb{S}_{r_{k}}$. So we can consider the assignment $(-)_{k}$, mapping every  Hoehnke radical $r$ to $r_{k}$. It worth noting that $(-)_{k}$ is order preserving. Indeed, if $r$ and $r'$ are two  Hoehnke radicals with $r\leq r'$ then we have
\[\begin{aligned}
& r\leq r' &
\Rightarrow & \left\{
\begin{matrix}
\mathbb{S}_{r'}\leq\mathbb{S}_{r}&\\
\mathbb{R}_{r}\leq\mathbb{R}_{r'}&
\end{matrix}\right.\\
& & \Rightarrow &\ r_{k}\leq r'_{k}.
\end{aligned}
\]
\end{remark}

 Now we show that the above Kurosh-Amitsur radical $r_{k}$ associated with a  Hoehnke radical $r_{h}$ enjoys the following properties.

\begin{proposition}\label{prop}
Let $r_{h}$ be a  Hoehnke radical and $r_{k}$ be the Kurosh-Amitsur radical determined by the radical class $\mathbb{R}_{r_{h}}$ of $r_{h}$. then
\begin{enumerate}
\item $r_{k}=\bigvee\{r\ | \ r \text{ is a Kurosh-Amitsur radical with } r\leq r_{h}\},$
\item $r_{k}=\bigwedge \{r \ | \ \mathbb{R}_{r}=\mathbb{R}_{r_{h}}, \text { and } r \text{ is a  Hoehnke radical}\}.$
\end{enumerate}
\end{proposition}

 \begin{proof}
First we note that for each $B\in \Sigma_{r_{k}(A)}$, there exists a subact $C$ of $A$ containing $B$ with $C\in \Sigma_{r_{h}(A)}$. Indeed, for each $B\in \Sigma_{r_{k}(A)} $, we take $C$ to be the greatest subact of $A$ with $\nabla_{C}\subseteq r_{h}(A)$ and show that $c/r_{h}(A)\subseteq C$, for every $c\in C$. For, otherwise, if there exists $a\in (c/r_{h}(A))\setminus C$, then $(sa, sc)\in r_{h}(A)$, for every $s\in S$. Since $sc\in C$ and $\nabla_{C}\subseteq r_{h}(A)$, transitivity of $r_{h}$ implies that $sa\in c/r_{h}(A)$, for every $s\in S$. So we have $C\subsetneq Sa\cup C$ and $\nabla_{Sa\cup C}\subseteq r_{h}(A)$, which contradict the choice of $C$.

 (1) To prove, first we show that $r_{k}(A)\leq r_{h}(A)$, for every $S$-act $A$. Indeed, for any $S$-act $A$ and $B\in \Sigma_{r_{k}(A)}$, we have $r_{k}(B)=\nabla_{B}$. Hence $B\in\mathbb{R}_{r_{k}}=\mathbb{R}_{r_{h}}$ and therefore, the definition of  Hoehnke radical implies that $\nabla_{B}=r_{h}(B)\leq r_{h}(A)$ and hence $\nabla_{B}=r_{h}(A)\wedge \nabla_{B}$. Then, there exists $C\in \Sigma_{r_{h}(A)}$ such that $B\leq C$. So $\Sigma_{r_{k}(A)}\leq \Sigma_{r_{h}(A)}$, which implies that $r_{k}\leq r_{h}$.\\
Now, we show that if $r(A)\leq r_{h}(A)$ then $r(A)\leq r_{k}(A)$, for each Kurosh-Amitsur radical $r$. Indeed, for a given Kurosh-Amitsur radical $r$, if $r\leq r_{h}$, then $\mathbb{R}_{r}\leq \mathbb{R}_{r_{h}} =\mathbb{R}_{r_{k}}$. Because, for any $S$-act $A$ and $B\in \Sigma_{r(A)}$, we have $r(B)=\nabla_{B}$, $B\in \mathbb{R}_{r}\subseteq \mathbb{R}_{r_{h}}=\mathbb{R}_{r_{k}}$. So there exists $C\in \Sigma _{r_{k}(A)}$ such that $B\subseteq C$ and this implies that $r(A)\leq r_{k}(A)$.

 (2) To show this part, first we show that $r_{k}(A)\leq r(A)$, for every  Hoehnke radical $r$ with $\mathbb{R}_{r}=\mathbb{R}_{r_{h}}$ and every $S$-act $A$. To do so, let $A$ be an $S$-act and $B\in \Sigma_{r_{k}(A)}$. Then, $r_{k}(B)=\nabla_{B}\in \mathbb{R}_{r_{k}} = \mathbb{R}_{r}$. Now, since $r$ is a  Hoehnke radical, we have $\nabla_{B}=r(B)\leq r(A)\wedge \nabla_{B} $ and hence $\nabla_{B}\leq r(A)$. Thus, there exists $C$ in $\Sigma _{r(A)}$ such that $B\subseteq C$. That is, $r_{k}(A)\leq r(A)$, for every  Hoehnke radical $r$ with $\mathbb{R}_{r}=\mathbb{R}_{r_{h}}$ and every $S$-act $A$. Also, since $r_{k}\in \{r\ | \ \mathbb{R}_{r}=\mathbb{R}_{r_{h}}, \text{ and } r \text{ is a  Hoehnke radical}\}$, $r_{k}$ is the greatest among the lower bounds of 
 \[\{r\ | \ \mathbb{R}_{r}=\mathbb{R}_{r_{h}}, \text { and } r \text { is a  Hoehnke radical} \}.\]
\end{proof}

 Now, by the above proposition we show that in the defintion of a Kurosh-Amitsur radical the third property can be omitted. See the following corollary.

 \begin{corollary}
In the definition of the Kurosh-Amitsur radicals, the properties {\rm({i})} and  {\rm(ii)} imply the third one.
\end{corollary}

 \begin{proof}
To prove, we show that the class $\overline {\mathcal{H}}$ of those  Hoehnke radicals which satisfy Properties {\rm({i})} and {\rm(ii)} is exactly the class of Kurosh-Amitsur radicals. But, it is clear that the class of Kurosh-Amitsur radicals is a subclass of  $\overline {\mathcal{H}}$. Now we show that every radical in $\overline{ \mathcal{H}}$ is a Kurosh-Amitsur radical. Indeed, for a given radical $r\in \overline{ \mathcal{H}}$, we consider the associated Kurosh-Amitsur radical $r_{k}$ with $r$, see Remark (\ref{r1}), with $\mathbb{R}_{r}=\mathbb{R}_{r_{k}}$. Then Proposition \ref{prop} implies $r\leq r_{k}$.
Thus $\Sigma_{r_{k}(A)}\leq \Sigma_{r(A)}$, for every $A\in~$\textbf{S-Act}. But since $r_k$ is a Kurosh-Amitsur radical, by the third property of the definition of Kurosh-Amitsur radical, we have $\Sigma_{r(A)}\leq \Sigma_{r_{k}(A)}$. Therefore $\Sigma_{r_{k}(A)}= \Sigma_{r(A)}$. Hence $r=r_{k}$ and we are done.
\end{proof}

 \begin{corollary}
The class of Kurosh-Amitsur radicals as a poset is a reflective subcategory of the class of  Hoehnke radicals.
\end{corollary}
\begin{proof}
We denote the class of the  Hoehnke radicals of {\bf S-Act} by $\mathcal{H}_{S}$ and the class of Kurosh-Amitsur radicals of {\bf S-Act} by $\mathcal{K}_{S}$. Then, since $\mathcal{H}_{S}$ forms a poset  with the  order
\[ r\leq r'\Leftrightarrow r(A)\leq r'(A),\ \text{ for every } A\in\text{{\bf S-Act}},\]
 one can consider $\mathcal{H}_{S}$ as a category. Obviously $\mathcal{K}_{S}$ is a full subcategory of $\mathcal{H}_{S}$. Also, since the map $(-)_{k}$, mapping every  Hoehnke radical $r$ to the Kurosh-Amitsur radical $r_{k}$, see Remark \ref{r1}, is order preserving, it can be regarded as a functor form $\mathcal{H}_{S}$ to $\mathcal{K}_{S}$. Now, Proposition 3.5 (1) indicates that $\mathcal{K}_{S}$ is a reflective full subcategory of $\mathcal{H}_{S}$, see Lemma 1.3.1 from \cite{Tholen95}.
\end{proof}

\end{document}